\documentclass[10pt]{amsart}

\usepackage{amsmath,amscd}
\usepackage{amssymb}
\usepackage{mathtools}
\usepackage{stmaryrd}
\usepackage{url}
\usepackage[all]{xy}

\usepackage{times}

\usepackage[T1]{fontenc}
\usepackage{graphicx}
\usepackage[svgnames]{xcolor}
\usepackage{framed}

\usepackage{tikz}

\usepackage{enumitem,kantlipsum}

\author{Liran Shaul}
\address{Department of Algebra, Faculty of Mathematics and Physics, Charles University in Prague, Sokolovsk\'a 83, 186 75 Praha, Czech Republic}

\email{shaul@karlin.mff.cuni.cz}
\newtheorem{thm}[equation]{Theorem}
\newtheorem*{thm*}{Theorem}
\newtheorem*{cor*}{Corollary}

\newtheorem{cor}[equation]{Corollary}
\newtheorem{prop}[equation]{Proposition}

\theoremstyle{definition}
\newtheorem{dfn}[equation]{Definition}
\newtheorem{rem}[equation]{Remark}

\newcommand{\opn}{\operatorname}
\newcommand{\cat}[1]{\operatorname{\mathsf{#1}}}

\newcommand{\mfrak}[1]{\mathfrak{#1}}

\newcommand{\mrm}[1]{\mathrm{#1}}

\renewcommand{\k}{\Bbbk}

\newcommand{\m}{\mfrak{m}}

\newcommand{\p}{\mfrak{p}}
\newcommand{\q}{\mfrak{q}}

\newcommand{\amp}{\operatorname{amp}}

\begin{document}

\title{Smooth flat maps over commutative DG-rings}

\begin{abstract}
We study smooth maps that arise in derived algebraic geometry.
Given a map $A \to B$ between non-positive commutative noetherian DG-rings which is of flat dimension $0$,
we show that it is smooth in the sense of To\"{e}n-Vezzosi if and only if it is homologically smooth in the sense of Kontsevich.
We then show that $B$, 
being a perfect DG-module over $B\otimes^{\mrm{L}}_A B$ has, locally, an explicit semi-free resolution as a Koszul complex.
As an application we show that a strong form of Van den Bergh duality between (derived) Hochschild homology and cohomology holds in this setting.
\end{abstract}


\setcounter{tocdepth}{1}
\setcounter{section}{-1}

\maketitle

\numberwithin{equation}{section}

\section{Introduction}

The aim of this paper is to study smooth maps that arise in derived algebraic geometry.
Since smoothness is a local condition, this study is affine in nature.
In classical commutative algebra, 
smooth maps between commutative noetherian rings are flat.
To generalize this situation, 
we will study here maps between commutative non-positive DG-rings which satisfy both a flatness condition,
as well as a smoothness condition.
There are two prominent smoothness conditions in the literature for maps between commutative DG-rings.
The first one is the notion of homological smoothness due to Kontsevich (see \cite{KS}).
Another notion, particularly important in derived algebraic geometry,
is the notion of smoothness in the sense of To\"{e}n-Vezzosi (\cite[Section 2.2.2]{TV}).
As the main result of this paper, we prove:

\begin{thm*}
Let $\varphi:A \to B$ be a flat\footnote{Here, by flat we mean a map of flat dimension $0$. See Section \ref{sec:fdz} for the precise definition.} map between commutative non-positive DG-rings,
such that the rings $\mrm{H}^0(A)$ and $\mrm{H}^0(B)$ are noetherian,
and the induced map $\mrm{H}^0(\varphi):\mrm{H}^0(A) \to \mrm{H}^0(B)$ is essentially of finite type.
Then the following are equivalent:
\begin{enumerate}
\item The map $\varphi$ is homologically smooth; i.e, the DG-ring $B$ is perfect over $B\otimes^{\mrm{L}}_A B$.
\item The map $\varphi$ is smooth in the sense of To\"{e}n-Vezzosi; i.e, $\varphi$ is flat and $\mrm{H}^0(\varphi)$ is smooth.
\item The diagonal map $\Delta:B\otimes^{\mrm{L}}_A B \to B$ is a local complete intersection:
for any prime ideal $\bar{\q} \in \opn{Spec}(\mrm{H}^0(B))$,
letting $\bar{\p} = \mrm{H}^0(\Delta)^{-1}(\bar{\q}) \in \opn{Spec}(\mrm{H}^0(B\otimes^{\mrm{L}}_A B))$,
there exist $\bar{a}_1,\dots,\bar{a}_n \in \mrm{H}^0\left((B\otimes^{\mrm{L}}_A B)_{\bar{\p}}\right)$ 
such that there is\footnote{To be precise, we show that the fact that a specific map being a quasi-isomorphism is equivalent to the other conditions. See Corollary \ref{cor:lci}.} a quasi-isomorphism
\[
B_{\bar{\q}} \to K((B\otimes^{\mrm{L}}_A B)_{\bar{\p}};\bar{a}_1,\dots,\bar{a}_n)
\]
where the right hand side is a Koszul complex over $(B\otimes^{\mrm{L}}_A B)_{\bar{\p}}$.
\end{enumerate}
\end{thm*}

This result is contained in Section \ref{sec:smooth} below, 
where several other conditions equivalent to smoothness are given,
with some of them making much weaker finiteness and noetherian assumptions.
We find item (3) above particularly surprising, 
as it indicates that $B$, being perfect over $B\otimes^{\mrm{L}}_A B$,
is locally given as a derived quotient of $B\otimes^{\mrm{L}}_A B$.
The results of Section \ref{sec:smooth} are actually particular cases of more general results we prove in Section \ref{sec:retract}. 
There, we study diagrams of commutative non-positive DG-rings of the form
\[
A \xrightarrow{\varphi} B \xrightarrow{\psi} A
\]
such that $\psi\circ \varphi = 1_A$. Such a diagram is called a DG-ring retraction.
We study in Section \ref{sec:retract} below DG-ring retractions with the extra condition that the map $\varphi$ is flat,
and show that in such a situation, properties of the map $\psi$ ascend and descend to properties of the surjective map $\mrm{H}^0(\psi)$.
We then apply these results, given a flat map $A \to B$ to diagrams of the form
\[
B \to B\otimes^{\mrm{L}}_A B \xrightarrow{\Delta} B
\]
which allows us to deduce the above theorem about smooth maps.
As an application of the above theorem, 
we obtain the following relation between derived Hochschild homology and derived Hochschild cohomology,
realizing Van den Bergh duality (\cite{VdB}) in this setting:

\begin{cor*}
Let $\varphi:A \to B$ be a flat map between commutative non-positive DG-rings,
such that the rings $\mrm{H}^0(A)$ and $\mrm{H}^0(B)$ are noetherian,
the ring $\mrm{H}^0(B)$ has connected spectrum,
and the induced map $\mrm{H}^0(\varphi):\mrm{H}^0(A) \to \mrm{H}^0(B)$ is essentially of finite type.
If $\varphi$ is smooth then there exist an invertible $\mrm{H}^0(B)$-module $\bar{M}$ and an integer $n$,
such that for all $i \in \mathbb{Z}$ and all $X \in \cat{D}(B\otimes^{\mrm{L}}_A B)$ there is a natural isomorphism
\[
\opn{Ext}^i_{B\otimes^{\mrm{L}}_A B}(B,X) \cong \bar{M}\otimes_{\mrm{H}^0(B)} \opn{Tor}_{n-i}^{B\otimes^{\mrm{L}}_A B}(B,X).
\]
\end{cor*}

This corollary is contained in Corollary \ref{cor:VDB}.
As a corollary of this result, we obtain an existence result about rigid DG-modules in this setting,
see Corollary \ref{cor:rigid}.

In the final Section \ref{sec:smooth-hz-flat} we study homologically smooth maps $\k \to A$,
such that $\k$ is a ring, and such that the induced map $\k \to \mrm{H}^0(A)$ is flat.
Among the results of that section, 
we deduce that under a noetherian assumption, 
commutative non-positive DG-rings are never homologically smooth over a field. 
Precisely:

\begin{thm*}
Let $\k$ be a field
and let $A$ be a non-positive homologically smooth commutative noetherian DG-ring over $\k$,
such that $A$ has bounded cohomology and $A^0$ is a noetherian ring.
Then $A$ is quasi-isomorphic to an ordinary regular ring.
\end{thm*}

\section{Preliminaries}

Throughout this paper we will be working with commutative non-positive DG-rings.
These are defined to be graded rings $A = \bigoplus_{n=-\infty}^0 A^n$,
with a $\mathbb{Z}$-linear differential $d:A \to A$ of degree $+1$.
Commutativity of $A$ means that for any homogeneous elements $a,b \in A$ there is an equality $b\cdot a = (-1)^{\deg(a)\cdot \deg(b)}\cdot a \cdot b$,
and $a^2 = 0$ if $\deg(a)$ is odd.
A Leibniz rule connects the multiplication of $A$ and the differential $d$:
\[
d(a\cdot b) = d(a)\cdot b + (-1)^{\deg(a)}\cdot a \cdot d(b).
\]
A complete reference for commutative non-positive DG-rings and derived categories over them is the book \cite{YeBook}.
For such a DG-ring $A$, 
we will denote by $\cat{D}(A)$ the unbounded derived category of DG-modules over $A$.
For $M \in \cat{D}(A)$, we let $\amp(M)$ denote the cohomological amplitude of $M$.
The full triangulated subcategory of $\cat{D}(A)$ which consists of DG-modules with bounded cohomology is denoted by $\cat{D}^{\mrm{b}}(A)$, that is, of all DG-modules $M$ such that $\amp(M) < \infty$,
while the full triangulated subcategory of $\cat{D}(A)$ which consists of DG-modules with bounded above cohomology will be denoted by $\cat{D}^{-}(A)$.
The set of degree $0$ elements of $A$, denoted by $A^0$ is a commutative ring.
Moreover, $\mrm{H}^0(A)$ is a quotient of $A^0$, and is also a commutative ring.
For any DG-module $M$, its cohomologies $\mrm{H}^i(M)$ are $\mrm{H}^0(A)$-modules.
We will say that $A$ is noetherian if $\mrm{H}^0(A)$ is a noetherian ring,
and for all $i<0$ the $\mrm{H}^0(A)$-module $\mrm{H}^i(A)$ is finitely generated.
Given two maps $A \to B$ and $A \to C$ of commutative non-positive DG-rings,
the derived tensor product $B\otimes^{\mrm{L}}_A C$ is also a commutative non-positive DG-ring.
To construct it, one factors the map $A \to B$ (or $A \to C$) as $A \to \widetilde{B} \to B$,
such that $A \to \widetilde{B}$ is K-flat, and $\widetilde{B} \to B$ is a quasi-isomorphism.
Then, one sets $B\otimes^{\mrm{L}}_A C := \widetilde{B} \otimes_A C$.
By working in the homotopy category of DG-rings, one may make this into a bifunctor.
In any case, up to a zig-zag of quasi-isomorphisms, 
the result is independent of the chosen resolution of $B$ (or $C$).

\subsection{Reduction of DG-modules}

Given a commutative non-positive DG-ring $A$,
by using the natural map $A \to \mrm{H}^0(A)$,
one obtains a functor $-\otimes^{\mrm{L}}_A \mrm{H}^0(A):\cat{D}^{-}(A)\to \cat{D}^{-}(\mrm{H}^0(A))$ which is sometimes called the reduction functor associated to $A$. The following basic properties of it will be useful in the sequel:

\begin{prop}\label{prop:reduction}
Let $A$ be a commutative non-positive DG-ring.
\begin{enumerate}
\item A morphism $f:M\to N$ in $\cat{D}^{-}(A)$ is an isomorphism if and only if its reduction $f\otimes^{\mrm{L}}_A \mrm{H}^0(A):M\otimes^{\mrm{L}}_A \mrm{H}^0(A) \to N\otimes^{\mrm{L}}_A \mrm{H}^0(A)$ in $\cat{D}^{-}(\mrm{H}^0(A))$ is an isomorphism.
\item Given $M \in \cat{D}^{-}(A)$, it holds that $M \cong A$ in $\cat{D}(A)$ if and only if $M\otimes^{\mrm{L}}_A \mrm{H}^0(A) \cong \mrm{H}^0(A)$ in $\cat{D}(\mrm{H}^0(A))$.
\end{enumerate}
\end{prop}
\begin{proof}
The first statement is \cite[Proposition 3.1]{Ye1}, 
and the second statement is \cite[Proposition 3.3(1)]{Ye1}.
\end{proof}

\subsection{Homomorphisms of flat dimension $0$}\label{sec:fdz}

We now recall a notion of flatness for maps between commutative non-positive DG-rings.
Following \cite[Section 2.2.1]{GR} and \cite[Section 2.2.2]{TV} we make the following definition:
\begin{dfn}
We say that a map $\varphi:A \to B$ between commutative non-positive DG-rings has flat dimension $0$,
if for any $M \in \cat{D}^{\mrm{b}}(A)$,
there is an inequality
\[
\amp(M\otimes^{\mrm{L}}_A B) \le \amp(M).
\]
\end{dfn}

\begin{rem}
Given a homomorphism $A \to B$ of flat dimension $0$,
it follows from the definition that, considered as an object of $\cat{D}(A)$,
the DG-module $B$ belongs to the class $\mathcal{F}$,
in the sense of \cite[Section 4.2]{Mi}.
In particular, by \cite[Lemma 4.6(4)]{Mi},
for any $M \in \cat{D}(A)$
and any $n \in \mathbb{Z}$,
there is a natural isomorphism
\begin{equation}\label{eqn:flat}
\mrm{H}^n(M\otimes^{\mrm{L}}_A B) \cong \mrm{H}^0(B) \otimes_{\mrm{H}^0(A)} \mrm{H}^n(M).
\end{equation}
It follows that there is an isomorphism of DG-rings
\begin{equation}\label{eqn:hz}
B \otimes^{\mrm{L}}_A \mrm{H}^0(A) \cong \mrm{H}^0(B).
\end{equation}
See \cite[Section 2.2.1.1]{GR}, \cite[Lemma 6.3]{Sh} and \cite[Section 2.2.2]{TV} for discussions about these facts.
\end{rem}

Here is an important fact about maps of flat dimension $0$:
\begin{prop}\label{prop:fdisflat}
Given a map $\varphi:A \to B$ of flat dimension $0$,
the induced map $\mrm{H}^0(\varphi):\mrm{H}^0(A) \to \mrm{H}^0(B)$ is flat.
\end{prop}
\begin{proof}
Given an $\mrm{H}^0(A)$-module $M$,
it follows from (\ref{eqn:hz}) and associativity of the derived tensor product that:
\[
\mrm{H}^0(B) \otimes^{\mrm{L}}_{\mrm{H}^0(A)} M \cong 
\left(B \otimes^{\mrm{L}}_A \mrm{H}^0(A)\right) \otimes^{\mrm{L}}_{\mrm{H}^0(A)} M \cong B\otimes^{\mrm{L}}_A M.
\]
Since $\mrm{H}^n(M) = 0$ for $n \ne 0$,
it follows from (\ref{eqn:flat}) that
\[
\mrm{H}^n\left(\mrm{H}^0(B) \otimes^{\mrm{L}}_{\mrm{H}^0(A)} M\right) \cong \mrm{H}^n\left(B\otimes^{\mrm{L}}_A M\right) \cong \mrm{H}^0(B) \otimes_{\mrm{H}^0(A)} \mrm{H}^n(M) = 0
\]
for all $n \ne 0$, so that $\mrm{H}^0(B)$ is flat over $\mrm{H}^0(A)$.
\end{proof}

Homomorphisms of flat dimension $0$ are stable under base change:
\begin{prop}\label{prop:flat-base-change}
Given a map $\varphi:A \to B$ of flat dimension $0$,
and a map $A \to C$,
the map $C \to B\otimes^{\mrm{L}}_A C$ is also of flat dimension $0$.
\end{prop}
\begin{proof}
Give $M \in \cat{D}^{\mrm{b}}(C)$, 
by the associativity of the derived tensor product:
\[
(B\otimes^{\mrm{L}}_A C) \otimes^{\mrm{L}}_C M \cong B\otimes^{\mrm{L}}_A M.
\]
Since $B$ has flat dimension $0$ over $A$, 
we know that $\amp(B\otimes^{\mrm{L}}_A M) \le \amp(M)$,
so that
\[
\amp\left((B\otimes^{\mrm{L}}_A C) \otimes^{\mrm{L}}_C M\right) = \amp(B\otimes^{\mrm{L}}_A M) \le \amp(M).
\]
\end{proof}

\subsection{Localization over commutative DG-rings}

Following \cite[Section 4]{Ye1},
given a commutative non-positive DG-ring $A$,
and given $\bar{\p} \in \opn{Spec}(\mrm{H}^0(A))$,
we define the localization $A_{\bar{\p}}$ as follows:
let $\p$ be the preimage of $\bar{\p}$ under the surjection $A^0 \to \mrm{H}^0(A)$,
and define $A_{\bar{\p}} := A \otimes_{A^0} A^0_{\p}$.
If $M$ is a DG-module over $A$,
we let $M_{\bar{\p}} := M\otimes^{\mrm{L}}_A A_{\bar{\p}}$.

\subsection{Koszul complexes and regular sequences over commutative DG-rings}

Given a commutative non-positive DG-ring $A$,
and given a finite sequence of elements $\bar{a}_1,\dots,\bar{a}_n \in \mrm{H}^0(A)$,
we define the Koszul complex over $A$ with respect to $\bar{a}_1,\dots,\bar{a}_n$ as follows:
for each $1\le i \le n$, 
choose some $a_i \in A^0$ whose image in $\mrm{H}^0(A)$ is equal to $\bar{a}_i$.
Consider $A$ as a DG-algebra over $\mathbb{Z}[x_1,\dots,x_n]$ by letting $x_i \mapsto a_i$,
and define
\[
K(A;\bar{a}_1,\dots,\bar{a}_n) := A\otimes^{\mrm{L}}_{\mathbb{Z}[x_1,\dots,x_n]} \mathbb{Z}.
\]
Equivalently, noting that 
\[
\mathbb{Z} \cong K(\mathbb{Z}[x_1,\dots,x_n];x_1,\dots,x_n),
\]
we may let
\[
K(A;\bar{a}_1,\dots,\bar{a}_n) := A\otimes_{\mathbb{Z}[x_1,\dots,x_n]} K(\mathbb{Z}[x_1,\dots,x_n];x_1,\dots,x_n).
\]
The Koszul complex $K(A;\bar{a}_1,\dots,\bar{a}_n)$ is also a commutative non-positive DG-ring,
and there is a natural map $\kappa:A \to K(A;\bar{a}_1,\dots,\bar{a}_n)$.
According to \cite[Lemma 2.8]{Mi2} or \cite[Proposition 2.6]{ShKos},
up to isomorphism in the homotopy category of DG-rings,
the Koszul complex is independent of the chosen lifts of $\bar{a}_1,\dots,\bar{a}_n$.
A commutative noetherian DG-ring $A$ is called local if the ring $\mrm{H}^0(A)$ is a local ring.
If $A$ is local, $\bar{\m}$ is the maximal ideal of $\mrm{H}^0(A)$,
and if $A$ has bounded cohomology,
then an element $\bar{a} \in \bar{\m}$ is called $A$-regular if it is $\mrm{H}^{\inf(A)}(A)$-regular,
that is, if the multiplication map 
\[
\bar{a}\cdot -:\mrm{H}^{\inf(A)}(A) \to \mrm{H}^{\inf(A)}(A)
\] 
is injective.
A sequence $\bar{a}_1,\dots,\bar{a}_n \in \bar{\m}$ is called $A$-regular,
if $\bar{a}_1$ is $A$-regular,
and the sequence $\bar{a}_2,\dots,\bar{a}_n$ is $K(A;\bar{a}_1)$-regular.

\subsection{Perfect and invertible DG-modules}

Following \cite[Sections 5,6]{Ye1} and \cite[Chapter 14]{YeBook},
a DG-module $M$ over a commutative non-positive DG-ring $A$
is called perfect if it belongs to the saturated full triangulated subcategory of $\cat{D}(A)$ generated by $A$.
This is equivalent to $M$ being a compact object of $\cat{D}(A)$.
We recall the following characterization of perfect DG-modules from \cite{Ye1}:
\begin{prop}\label{prop:compact}
Let $A$ be a commutative non-positive DG-ring,
and let $M \in \cat{D}^{-}(A)$.
Then $M$ is perfect over $A$ if and only if $M\otimes^{\mrm{L}}_A \mrm{H}^0(A)$ is perfect over $\mrm{H}^0(A)$.
\end{prop}
\begin{proof}
This is contained in \cite[Theorem 5.11]{Ye1}.
\end{proof}

A DG-module $M\in \cat{D}^{-}(A)$ is called invertible (these are called tilting DG-modules in \cite[Chapter 14]{YeBook}) if there exist a DG-module $N\in \cat{D}^{-}(A)$,
such that $M\otimes^{\mrm{L}}_A N \cong A$.
In that case, by \cite[Proposition 14.2.19]{YeBook}, it follows that $M$ is a perfect DG-module. 
The DG-module $N$ is then called the inverse of $M$, and
according to \cite[Corollary 14.4.27]{YeBook},
it is given by $N\cong \mrm{R}\opn{Hom}_A(M,A)$.

\begin{prop}\label{prop:red-is-invertible}
Let $A$ be a commutative non-positive DG-ring,
and let $M \in \cat{D}^{-}(A)$.
Then $M$ is invertible over $A$ if and only if $M\otimes^{\mrm{L}}_A \mrm{H}^0(A)$ is invertible over $\mrm{H}^0(A)$.
\end{prop}
\begin{proof}
If $M$ is invertible over $A$ with an inverse $N$,
the isomorphism 
\[
(M\otimes^{\mrm{L}}_A \mrm{H}^0(A)) \otimes^{\mrm{L}}_{\mrm{H}^0(A)} (N\otimes^{\mrm{L}}_A \mrm{H}^0(A)) \cong (M\otimes^{\mrm{L}}_A N) \otimes^{\mrm{L}}_A \mrm{H}^0(A)
\]
shows that $M\otimes^{\mrm{L}}_A \mrm{H}^0(A)$ is invertible over $\mrm{H}^0(A)$ with inverse being $N\otimes^{\mrm{L}}_A \mrm{H}^0(A)$.
Conversely, suppose that $M \in \cat{D}^{-}(A)$ satisfies that $M\otimes^{\mrm{L}}_A \mrm{H}^0(A)$ is invertible.
In particular, $M\otimes^{\mrm{L}}_A \mrm{H}^0(A)$ is perfect over $\mrm{H}^0(A)$, so by Proposition \ref{prop:compact}, 
it follows that $M$ is perfect over $A$.
By \cite[Theorem 14.1.22]{YeBook},
the natural map
\[
\mrm{R}\opn{Hom}_A(M,A) \otimes^{\mrm{L}}_A M \to \mrm{R}\opn{Hom}_A(M,M)
\]
is an isomorphism. 
Using \cite[Theorem 14.1.22]{YeBook} again, and by adjunction,
we obtain the following isomorphisms in $\cat{D}(\mrm{H}^0(A))$:
\begin{gather*}
\mrm{R}\opn{Hom}_A(M,M) \otimes^{\mrm{L}}_A \mrm{H}^0(A) \cong
\mrm{R}\opn{Hom}_A(M,M \otimes^{\mrm{L}}_A \mrm{H}^0(A)) \cong\\
\mrm{R}\opn{Hom}_{\mrm{H}^0(A)}(M \otimes^{\mrm{L}}_A \mrm{H}^0(A),M \otimes^{\mrm{L}}_A \mrm{H}^0(A)) \cong \\
\mrm{R}\opn{Hom}_{\mrm{H}^0(A)}(M \otimes^{\mrm{L}}_A \mrm{H}^0(A),\mrm{H}^0(A)) \otimes^{\mrm{L}}_{\mrm{H}^0(A)} (M \otimes^{\mrm{L}}_A \mrm{H}^0(A)) \cong \mrm{H}^0(A).
\end{gather*}
This implies by Proposition \ref{prop:reduction}(2) that $\mrm{R}\opn{Hom}_A(M,M) \cong A$, so that $M$ is invertible over $A$.
\end{proof}

If $A$ is an ordinary ring and $M$ is an $A$-module,
then it is well known that $M$ is invertible if and only if it is projective of rank $1$,
if and only if for all $\p \in \opn{Spec}(A)$ there is an isomorphism $M_{\p} \cong A_{\p}$.
Similarly, we have in the DG-setting:

\begin{prop}\label{prop:tiltingLocally}
Let $A$ be a commutative non-positive DG-ring,
let $M \in \cat{D}^{-}(A)$ be a bounded above DG-module,
and suppose that for any $\bar{\p} \in \opn{Spec}(\mrm{H}^0(A))$,
there exist $n \in \mathbb{Z}$ such that $M_{\bar{\p}} \cong A_{\bar{\p}}[n]$.
Then $M$ is an invertible DG-module over $A$.
\end{prop}
\begin{proof}
Letting $\bar{M} := M\otimes^{\mrm{L}}_A \mrm{H}^0(A)$,
we have that for $\bar{\p} \in \opn{Spec}(\mrm{H}^0(A))$ there exist an $n \in \mathbb{Z}$,
such that $\bar{M}_{\bar{\p}} \cong \mrm{H}^0(A)_{\bar{\p}}[n]$.
This shows (for instance, by \cite[tag 0FNT]{SP}),
that $\bar{M}$ is invertible over $\mrm{H}^0(A)$,
so by Proposition \ref{prop:red-is-invertible},
this implies that $M$ is invertible over $A$.
\end{proof}

\section{DG-ring retractions of flat dimension $0$}\label{sec:retract}

A diagram of commutative non-positive of DG-rings of the form
\begin{equation}\label{eqn:retract}
A \xrightarrow{\varphi} B \xrightarrow{\psi} A
\end{equation}
such that $\psi \circ \varphi = 1_A$ is called a DG-ring retraction.
The key example to keep in mind is,
given a map $A \to B$,
it give rise to the DG-ring retraction
\[
B \to B\otimes^{\mrm{L}}_A B \to B.
\]
We now study retractions of the form (\ref{eqn:retract}) such that $A \xrightarrow{\varphi} B$ is of flat dimension $0$.

\begin{thm}\label{thm:smoothRed}
Given a DG-ring retraction of commutative non-positive DG-rings
\[
A \xrightarrow{\varphi} B \xrightarrow{\psi} A
\]
such that $\varphi$ has flat dimension $0$,
it holds that $A$ is perfect over $B$ if and only if $\mrm{H}^0(A)$ is perfect over $\mrm{H}^0(B)$.
\end{thm}
\begin{proof}
According to Proposition \ref{prop:compact}, 
we have that $A$ is perfect over $B$ if and only if
\[
A \otimes^{\mrm{L}}_B \mrm{H}^0(B)
\]
is perfect over $\mrm{H}^0(B)$.
Using (\ref{eqn:hz}), we have the following sequence of isomorphisms in $\cat{D}(A)$:
\[
A \otimes^{\mrm{L}}_B \mrm{H}^0(B) \cong A \otimes^{\mrm{L}}_B \left(B\otimes^{\mrm{L}}_A \mrm{H}^0(A)\right) \cong \mrm{H}^0(A).
\]
Since the forgetful functor $\cat{D}(\mrm{H}^0(B)) \to \cat{D}(A)$ is conservative and commutes with cohomology,
it follows that as an object of $\cat{D}(\mrm{H}^0(B))$,
the complex $A \otimes^{\mrm{L}}_B \mrm{H}^0(B)$ satisfies 
\[
\mrm{H}^n(A \otimes^{\mrm{L}}_B \mrm{H}^0(B)) = 0
\]
for all $n \ne 0$.
This implies by \cite[Proposition 7.3.10]{YeBook}, 
that the truncation map 
\[
A \otimes^{\mrm{L}}_B \mrm{H}^0(B) \to \mrm{H}^0(A \otimes^{\mrm{L}}_B \mrm{H}^0(B))
\]
is an isomorphism in $\cat{D}(\mrm{H}^0(B))$.
Since $A$ is non-positive, we know that there is a $\mrm{H}^0(B)$-linear isomorphism
\[
\mrm{H}^0(A \otimes^{\mrm{L}}_B \mrm{H}^0(B)) \cong \mrm{H}^0(A) \otimes_B \mrm{H}^0(B) \cong \mrm{H}^0(A).
\]
Combining the above two isomorphisms, 
we see that there is an isomorphism 
\[
A \otimes^{\mrm{L}}_B \mrm{H}^0(B) \cong \mrm{H}^0(A)
\]
in $\cat{D}(\mrm{H}^0(B))$.
This shows that $A$ is perfect over $B$ if and only if $\mrm{H}^0(A)$ is perfect over $\mrm{H}^0(B)$.
\end{proof}

Next, we wish to discuss a situation where for such a retraction, 
when $A$ is perfect over $B$,
it has, locally, an explicit semi-free resolution which demonstrates its perfectness.
Notice that given a retraction of DG-rings
\[
A \xrightarrow{\varphi} B \xrightarrow{\psi} A
\]
applying the functor $\mrm{H}^0$ to it give rise to a retraction
\[
\mrm{H}^0(A) \xrightarrow{\mrm{H}^0(\varphi)} \mrm{H}^0(B) \xrightarrow{\mrm{H}^0(\psi)} \mrm{H}^0(A)
\]
of commutative rings. 
In particular it follows that the map $\mrm{H}^0(\psi):\mrm{H}^0(B) \to \mrm{H}^0(A)$ is surjective.

Recall that a surjective homomorphism $\psi:B \to A$ between commutative noetherian rings is called a complete intersection homomorphism if the ideal $\ker(\psi)$ is generated by a $B$-regular sequence.
A surjection $\psi:B \to A$ is called a locally complete intersection if for any $\q \in \opn{Spec}(A)$,
letting $\p = \psi^{-1}(\q)$, the induced map $\psi_{\p}:B_{\p} \to A_{\q}$ is a complete intersection.
Equivalently, the ideal $\ker(\psi_{\p}) \subseteq B_{\p}$ is generated by a $B_{\p}$-regular sequence.
If $a_1,\dots,a_n$ is a $B_{\p}$-regular sequence such that $\ker(\psi_{\p}) = (a_1,\dots,a_n)$,
we deduce that 
\begin{equation}\label{eq:koszulis}
A_{\q} \cong B_{\p}/\ker(\psi_{\p}) \cong K(B_{\p};a_1,\dots,a_n),
\end{equation}
where the latter denotes the Koszul complex over $B_{\p}$ with respect to $a_1,\dots,a_n$.
If moreover the map $\psi$ is part of a retraction diagram
\[
A \xrightarrow{\varphi} B \xrightarrow{\psi} A
\]
then we may realize the isomorphism (\ref{eq:koszulis}) as the composition
\[
A_{\q} \xrightarrow{\varphi_{\p}} B_{\p} \xrightarrow{\kappa} K(B_{\p};a_1,\dots,a_n)
\]
where $\kappa$ is the canonical map from a ring to the Koszul complex over it.

To generalize this to the DG-setting, 
we propose the following definition:
\begin{dfn}
Let $\psi:B \to A$ be a map between commutative non-positive DG-rings,
such that $\mrm{H}^0(B)$ and $\mrm{H}^0(A)$ are noetherian,
and such that $\mrm{H}^0(\psi)$ is surjective.
\begin{enumerate}
\item We say that $\psi$ is a complete intersection if there exist 
$\bar{a}_1,\dots,\bar{a}_n \in \ker\left(\mrm{H}^0(\psi)\right)$ such that there is a quasi-isomorphism
$A \cong K(B;\bar{a}_1,\dots,\bar{a}_n)$.
\item We say that $\psi$ is a local complete intersection if for any $\bar{\q} \in \opn{Spec}(\mrm{H}^0(A))$,
letting $\bar{\p} = \mrm{H}^0(\psi)^{-1}(\bar{\q})$,
there exist $\bar{a}_1,\dots,\bar{a}_n \in \ker\left(\mrm{H}^0(\psi_{\p})\right)$ such that there is a quasi-isomorphism
$A_{\bar{\q}} \cong K(B_{\bar{\p}};\bar{a}_1,\dots,\bar{a}_n)$.
\end{enumerate}
\end{dfn}

\begin{thm}\label{thm:lciRed}
Assume there is a DG-ring retraction of commutative non-positive DG-rings
\[
A \xrightarrow{\varphi} B \xrightarrow{\psi} A
\]
such that $\varphi$ has flat dimension $0$,
and the rings $\mrm{H}^0(A)$ and $\mrm{H}^0(B)$ are noetherian.
Then the following are equivalent:
\begin{enumerate}
\item The map $\mrm{H}^0(\psi)$ is a local complete intersection.
\item The map $\psi$ is a local complete intersection,
such that for any $\bar{\q} \in \opn{Spec}(\mrm{H}^0(A))$,
letting $\bar{\p} = \mrm{H}^0(\psi)^{-1}(\bar{\q})$,
the quasi-isomorphism $A_{\bar{\q}} \cong K(B_{\bar{\p}};\bar{a}_1,\dots,\bar{a}_n)$ 
is given by the composition
\[
A_{\bar{\q}} \xrightarrow{\varphi_{\bar{\p}}} B_{\bar{\p}} \xrightarrow{\kappa} K(B_{\bar{\p}};\bar{a}_1,\dots,\bar{a}_n).
\]
\end{enumerate}
\end{thm}
\begin{proof}
Suppose that $\mrm{H}^0(\psi)$ is a local complete intersection.
Given any $\bar{\q} \in \opn{Spec}(\mrm{H}^0(A))$,
letting $\bar{\p} = \mrm{H}^0(\psi)^{-1}(\bar{\q})$,
take a $\mrm{H}^0(B)$-regular sequence $\bar{a}_1,\dots,\bar{a}_n \in \ker\left(\mrm{H}^0(\psi_{\bar{\p}})\right)$ 
such that $\ker\left(\mrm{H}^0(\psi_{\bar{\p}})\right) = (\bar{a}_1,\dots,\bar{a}_n)$.
Since localization is exact, the map $\varphi_{\bar{\p}}$ is still of flat dimension $0$.
Consider the diagram of maps of DG-rings:
\begin{equation}\label{eq:diag-before-tens}
A_{\bar{\q}} \xrightarrow{\varphi_{\bar{\p}}} B_{\bar{\p}} \xrightarrow{\kappa} K(B_{\bar{\p}};\bar{a}_1,\dots,\bar{a}_n).
\end{equation}
We note that these two maps are $A_{\bar{\q}}$-linear:
for $\varphi_{\bar{\p}}$ this is clear, while $\kappa$ can be obtained by applying the functor $-\otimes_{\mathbb{Z}[x_1,\dots,x_n]} B_{\bar{\p}}$ to the map 
\[
\mathbb{Z}[x_1,\dots,x_n] \to K(\mathbb{Z}[x_1,\dots,x_n];x_1,\dots,x_n),
\]
and hence $\kappa$ is $B_{\bar{\p}}$-linear, so a fortiori it is $A_{\bar{\q}}$-linear.
Hence, to show that $\kappa \circ \varphi_{\bar{\p}}$ is a quasi-isomorphism,
we may consider it as a morphism in $\cat{D}(A_{\bar{\q}})$,
and it is enough to show that in $\cat{D}(A_{\bar{\q}})$ it is an isomorphism.
Since all these DG-rings are non-positive,
by Proposition \ref{prop:reduction}(1), it is enough to show that the map
\[
\left(\kappa \circ \varphi_{\bar{\p}}\right) \otimes^{\mrm{L}}_{A_{\bar{\q}}} \mrm{H}^0(A_{\bar{\q}}) :
A_{\bar{\q}} \otimes^{\mrm{L}}_{A_{\bar{\q}}} \mrm{H}^0(A_{\bar{\q}}) \to K(B_{\bar{\p}};\bar{a}_1,\dots,\bar{a}_n) \otimes^{\mrm{L}}_{A_{\bar{\q}}} \mrm{H}^0(A_{\bar{\q}})
\]
is an isomorphism.
Using (\ref{eqn:hz}) 
the application of the functor $-\otimes^{\mrm{L}}_{A_{\bar{\q}}} \mrm{H}^0(A_{\bar{\q}})$ 
to the map $\varphi_{\bar{\p}}:A_{\bar{\q}} \to B_{\bar{\p}}$
yields the map
\begin{equation}\label{eq:diag-after-tens}
\mrm{H}^0(A_{\bar{\q}}) \xrightarrow{\mrm{H}^0(\varphi_{\bar{\p}})} \mrm{H}^0(B_{\bar{\p}}) 
\end{equation}
On the other hand, by using (\ref{eqn:hz}) and \cite[Proposition 2.9]{ShKos}, we have the following isomorphisms in $\cat{D}(A_{\bar{\q}})$:
\begin{gather*}
K(B_{\bar{\p}};\bar{a}_1,\dots,\bar{a}_n) \otimes^{\mrm{L}}_{A_{\bar{\q}}} \mrm{H}^0(A_{\bar{\q}}) \cong
K(B_{\bar{\p}};\bar{a}_1,\dots,\bar{a}_n) \otimes^{\mrm{L}}_{B_{\bar{\p}}} B_{\bar{\p}} \otimes^{\mrm{L}}_{A_{\bar{\q}}} \mrm{H}^0(A_{\bar{\q}}) \cong\\
K(B_{\bar{\p}};\bar{a}_1,\dots,\bar{a}_n) \otimes^{\mrm{L}}_{B_{\bar{\p}}} \mrm{H}^0(B_{\bar{\p}}) \cong
K(\mrm{H}^0(B_{\bar{\p}});\bar{a}_1,\dots,\bar{a}_n).
\end{gather*}
Note that by our regularity assumption on the sequence $\bar{a}_1,\dots,\bar{a}_n$,
it holds that the DG-module $K(\mrm{H}^0(B_{\bar{\p}});\bar{a}_1,\dots,\bar{a}_n)$ has non-zero cohomology only at degree $0$,
as is the case for the DG-modules in (\ref{eq:diag-after-tens}). 
But for such DG-modules, the functors $-\otimes^{\mrm{L}}_{A_{\bar{\q}}} \mrm{H}^0(A_{\bar{\q}})$  
and $\mrm{H}^0(-\otimes^{\mrm{L}}_{A_{\bar{\q}}} \mrm{H}^0(A_{\bar{\q}}))$ are isomorphic.
It follows that the diagram obtained from applying the functor $-\otimes^{\mrm{L}}_{A_{\bar{\q}}} \mrm{H}^0(A_{\bar{\q}})$ 
to the diagram (\ref{eq:diag-before-tens}) is isomorphic in $\cat{D}(A_{\bar{\q}})$ to the diagram
\begin{equation}\label{eq:diag-after-tens-final}
\mrm{H}^0(A_{\bar{\q}}) \to \mrm{H}^0(B_{\bar{\p}}) \to K(\mrm{H}^0(B_{\bar{\p}});\bar{a}_1,\dots,\bar{a}_n).
\end{equation}

By our assumption on $\mrm{H}^0(\psi)$,
the composition (\ref{eq:diag-after-tens-final}) is an isomorphism,
which implies that (\ref{eq:diag-before-tens}) is also an isomorphism,
so that $\psi$ is a local complete intersection.
Conversely, if there are $\bar{a}_1,\dots,\bar{a}_n \in \ker\left(\mrm{H}^0(\psi_{\bar{\p}})\right)$
such that the composition
\[
A_{\bar{\q}} \xrightarrow{\varphi_{\bar{\p}}} B_{\bar{\p}} \xrightarrow{\kappa} K(B_{\bar{\p}};\bar{a}_1,\dots,\bar{a}_n)
\]
is a quasi-isomorphism,
applying $-\otimes^{\mrm{L}}_{A_{\bar{\q}}} \mrm{H}^0(A_{\bar{\q}})$ to it gives the diagram
\[
\mrm{H}^0(A_{\bar{\q}}) \to \mrm{H}^0(B_{\bar{\p}}) \to K(\mrm{H}^0(B_{\bar{\p}});\bar{a}_1,\dots,\bar{a}_n).
\]
in which the composition is still a quasi-isomorphism,
which implies that $\bar{a}_1,\dots,\bar{a}_n$ is a $\mrm{H}^0(B_{\bar{\p}})$-regular sequence,
so that $\mrm{H}^0(\psi)$ is a local complete intersection.
\end{proof}

Under a stronger noetherian assumption, we can say slightly more:
\begin{prop}\label{prop:isRegular}
In the situation of Theorem \ref{thm:lciRed},
assume also that the DG-rings $A$ and $B$ are noetherian and have bounded cohomology.
Then the sequence $\bar{a}_1,\dots,\bar{a}_n$ is a $B_{\bar{\p}}$-regular sequence.
\end{prop}
\begin{proof}
Since the map $A_{\bar{\q}} \to B_{\bar{\p}}$ is of flat dimension $0$,
we know by (\ref{eqn:flat}) that
\[
\mrm{H}^n(B_{\bar{\p}}) \cong \mrm{H}^0(B_{\bar{\p}}) \otimes_{\mrm{H}^0(A_{\bar{\q}})} \mrm{H}^n(A_{\bar{\q}}),
\]
so that $\amp(B_{\bar{\p}}) \le \amp(A_{\bar{\q}})$.
On the other hand, the fact that
\[
A_{\bar{\q}} \to B_{\bar{\p}} \to A_{\bar{\q}}
\]
is a retraction shows that applying the functor $\mrm{H}^n$ gives a retraction
\[
\mrm{H}^n(A_{\bar{\q}}) \to \mrm{H}^n(B_{\bar{\p}}) \to \mrm{H}^n(A_{\bar{\q}})
\]
which implies that $\amp(B_{\bar{\p}}) \ge \amp(A_{\bar{\q}})$.
It follows that there are equalities
\[
\amp(B_{\bar{\p}}) = \amp(A_{\bar{\q}}) = \amp\left(K(B_{\bar{\p}};\bar{a}_1,\dots,\bar{a}_n)\right).
\]
By \cite[Lemma 2.13(2)]{Mi2}, this implies that $\bar{a}_1,\dots,\bar{a}_n$ is a $B_{\bar{\p}}$-regular sequence.
\end{proof}

Finally, we combine the local complete intersection and perfectness assumptions,
and obtain the following variant of Van den Bergh duality (see \cite[Theorem 1]{VdB}):

\begin{thm}\label{thm:VDB}
Assume there is a DG-ring retraction of commutative non-positive DG-rings
\[
A \xrightarrow{\varphi} B \xrightarrow{\psi} A
\]
Suppose that:
\begin{enumerate}
\item The map $\varphi$ has flat dimension $0$.
\item The rings $\mrm{H}^0(A)$ and $\mrm{H}^0(B)$ are noetherian,
and the map $\mrm{H}^0(\psi)$ is a local complete intersection.
\item The DG-ring $A$ is perfect over $B$.
\end{enumerate}
Then the following holds:
\begin{enumerate}
\item The DG-module $\mrm{R}\opn{Hom}_B(A,B) \in \cat{D}(A)$ is an invertible DG-module.
\item Assume further that the noetherian ring $\mrm{H}^0(A)$ has connected spectrum.
Then there exist an $n \in \mathbb{N}$,
and an invertible $\mrm{H}^0(A)$-module $\bar{N}$ such that for any $X \in \cat{D}(B)$ and any $i \in \mathbb{Z}$
there is a natural isomorphism
\[
\mrm{H}^i \left(\mrm{R}\opn{Hom}_B(A,X)\right) \cong 
\bar{N}\otimes_{\mrm{H}^0(A)} \mrm{H}^{i-n} \left( A\otimes^{\mrm{L}}_B X \right).
\]
\end{enumerate}
\end{thm}
\begin{proof}
Let us denote by $\varphi_*:\cat{D}(B)\to\cat{D}(A)$ and $\psi_*:\cat{D}(A)\to\cat{D}(B)$ the forgetful functors along $\varphi$ and $\psi$ respectively.
Let us set 
\[
N = \mrm{R}\opn{Hom}_B(A,B) \in \cat{D}(B),
\]
and let $M=\varphi_*(N)\in \cat{D}(A)$.
Given any $\bar{\q} \in \opn{Spec}(\mrm{H}^0(A))$,
let $\bar{\p} = \mrm{H}^0(\psi)^{-1}(\bar{\q})$.
We will calculate the localization $M_{\bar{\q}}$.
We note that there is an isomorphism $A\otimes^{\mrm{L}}_B B_{\bar{\p}} \cong A_{\bar{\q}}$ in $\cat{D}(A)$,
so the functors $-\otimes^{\mrm{L}}_A A_{\bar{\q}}$ and $\psi_*(-)\otimes^{\mrm{L}}_B B_{\bar{\p}}$,
considered as functors $\cat{D}(A) \to \cat{D}(A)$ are isomorphic.
Using this we have the following isomorphisms in $\cat{D}(A)$:
\[
M_{\bar{\q}} \cong \varphi_*(N)\otimes^{\mrm{L}}_A A_{\bar{\q}} \cong \psi_*(\varphi_*(N))\otimes^{\mrm{L}}_B B_{\bar{\p}}.
\]
We claim that the DG-modules $N$ and $\psi_*(\varphi_*(N))$ in $\cat{D}(B)$ are isomorphic.
Since the DG-module $\psi_*(\varphi_*(N))$ is obtained from $N$ by applying forgetful functors, 
we only need to verify that the $B$-actions on these two DG-modules coincide. 
To see this, let us take a K-injective resolution $B\cong I$ in $\cat{D}(B)$, 
and then calculate $N \cong \opn{Hom}_B(A,I)$.
Given $f \in \opn{Hom}_B(A,I)^i$, and given $b\in B^j$ and some $a\in A$, 
we have that 
\[
(b\cdot f)(a) = b\cdot (f(a)) = (-1)^{i\cdot j} f(b\cdot a) = (-1)^{i\cdot j} f(\psi(b)\cdot a)
\]
where the last equality follows from the fact that $B$ acts on $A$ via $\psi$.
On the other hand, given $f' \in \psi_*(\varphi_*(\opn{Hom}_B(A,I)))^i$,
denoting by $f \in \opn{Hom}_B(A,I)^i$ the corresponding element before applying these forgetful functors,
and given $b\in B^j$ and $a\in A$, we obtain
\[
(b\cdot f')(a) = (\varphi(\psi(b)) \cdot f)(a) = (-1)^{i\cdot j} f(\psi(\varphi(\psi(b)))\cdot a) = (-1)^{i\cdot j} f(\psi(b)\cdot a)
\]
where here we used the fact that $\psi \circ \varphi = 1_A$.
This shows that $\psi_*(\varphi_*(N)) \cong N$.
It follows that to compute $M_{\bar{\q}}$, 
we may instead compute 
$N\otimes^{\mrm{L}}_B B_{\bar{\p}}$.
By \cite[Theorem 14.1.22]{YeBook}, which holds since $A$ is perfect over $B$,
\[
N\otimes^{\mrm{L}}_B B_{\bar{\p}}
\cong
\mrm{R}\opn{Hom}_B(A,B) \otimes^{\mrm{L}}_B B_{\bar{\p}} \cong
\mrm{R}\opn{Hom}_B(A,B_{\bar{\p}}).
\]
By adjunction we have:
\[
\mrm{R}\opn{Hom}_B(A,B_{\bar{\p}}) \cong 
\mrm{R}\opn{Hom}_{B_{\bar{\p}}}(A\otimes^{\mrm{L}}_B B_{\bar{\p}},B_{\bar{\p}}) \cong 
\mrm{R}\opn{Hom}_{B_{\bar{\p}}}(A_{\bar{\q}},B_{\bar{\p}}).
\]
To prove that $M$ is invertible, by Proposition \ref{prop:tiltingLocally},
it is enough to show that 
\begin{equation}\label{eqn:goalTilting}
M_{\bar{\q}} \cong \mrm{R}\opn{Hom}_{B_{\bar{\p}}}(A_{\bar{\q}},B_{\bar{\p}}) \stackrel{?}{\cong} A_{\bar{\q}}[m]
\end{equation}
for some integer $m$.
To show this, it is tempting to use Theorem \ref{thm:lciRed} which implies that there exist $\bar{a}_1,\dots,\bar{a}_n \in \mrm{H}^0(B_{\bar{\p}})$,
such that $A_{\bar{\q}} \stackrel{(\diamond)}{\cong} K(B_{\bar{\p}};\bar{a}_1,\dots,\bar{a}_n)$.
However, the isomorphism $(\diamond)$ is only $A_{\bar{\q}}$-linear, 
so does not help us to compute the above.
Instead, we will calculate:
\begin{gather*}
\mrm{R}\opn{Hom}_{B_{\bar{\p}}}(A_{\bar{\q}},B_{\bar{\p}})\otimes^{\mrm{L}}_{A_{\bar{\q}}} \mrm{H}^0(A_{\bar{\q}}) \cong\\
\mrm{R}\opn{Hom}_{B_{\bar{\p}}}(A_{\bar{\q}},B_{\bar{\p}} \otimes^{\mrm{L}}_{A_{\bar{\q}}} \mrm{H}^0(A_{\bar{\q}})) \cong\\
\mrm{R}\opn{Hom}_{B_{\bar{\p}}}(A_{\bar{\q}},\mrm{H}^0(B_{\bar{\p}})). 
\end{gather*}
Here, the first isomorphism is by \cite[Theorem 14.1.22]{YeBook}, which holds since $A_{\bar{\q}}$ is perfect over $B_{\bar{\p}}$,
and the second isomorphism is by (\ref{eqn:hz}).
By adjunction, we have that
\begin{gather*}
\mrm{R}\opn{Hom}_{B_{\bar{\p}}}(A_{\bar{\q}},\mrm{H}^0(B_{\bar{\p}})) \cong
\mrm{R}\opn{Hom}_{\mrm{H}^0(B_{\bar{\p}})}(A_{\bar{\q}}\otimes^{\mrm{L}}_{B_{\bar{\p}}} \mrm{H}^0(B_{\bar{\p}}),\mrm{H}^0(B_{\bar{\p}})).
\end{gather*}
As in the proof of Theorem \ref{thm:smoothRed}, 
we know that there is an isomorphism $A_{\bar{\q}}\otimes^{\mrm{L}}_{B_{\bar{\p}}} \mrm{H}^0(B_{\bar{\p}}) \cong \mrm{H}^0(A_{\bar{\q}})$
in $\cat{D}(\mrm{H}^0(B_{\bar{\p}}))$,
so we see that
\[
\mrm{R}\opn{Hom}_{\mrm{H}^0(B_{\bar{\p}})}(A_{\bar{\q}}\otimes^{\mrm{L}}_{B_{\bar{\p}}} \mrm{H}^0(B_{\bar{\p}}),\mrm{H}^0(B_{\bar{\p}})) \cong 
\mrm{R}\opn{Hom}_{\mrm{H}^0(B_{\bar{\p}})}(\mrm{H}^0(A_{\bar{\q}}),\mrm{H}^0(B_{\bar{\p}})).
\]
Since the map $\mrm{H}^0(\psi_{\bar{\p}})$ is a complete intersection map,
there exist a $\mrm{H}^0(B_{\bar{\p}})$-regular sequence $\bar{a}_1,\dots,\bar{a}_n \in \mrm{H}^0(B_{\bar{\p}})$ 
such that $\mrm{H}^0(A_{\bar{\q}}) \cong K(\mrm{H}^0(B_{\bar{\p}});\bar{a}_1,\dots,\bar{a}_n)$.
Note that this isomorphism, which is between two complexes with cohomology is concentrated in degree $0$ holds in $\cat{D}(\mrm{H}^0(B_{\bar{\p}}))$.
Hence, by the self-duality of the Koszul complex, we see that
\begin{gather*}
\mrm{R}\opn{Hom}_{\mrm{H}^0(B_{\bar{\p}})}(\mrm{H}^0(A_{\bar{\q}}),\mrm{H}^0(B_{\bar{\p}}))
\cong\\
\mrm{R}\opn{Hom}_{\mrm{H}^0(B_{\bar{\p}})}(K(\mrm{H}^0(B_{\bar{\p}});\bar{a}_1,\dots,\bar{a}_n),\mrm{H}^0(B_{\bar{\p}}))
\cong\\
K(\mrm{H}^0(B_{\bar{\p}});\bar{a}_1,\dots,\bar{a}_n)[-n] \cong \mrm{H}^0(A_{\bar{\q}})[-n].
\end{gather*}
We have thus obtained an isomorphism
\[
\mrm{R}\opn{Hom}_{B_{\bar{\p}}}(A_{\bar{\q}},B_{\bar{\p}})\otimes^{\mrm{L}}_{A_{\bar{\q}}} \mrm{H}^0(A_{\bar{\q}}) \cong \mrm{H}^0(A_{\bar{\q}})[-n]
\]
in $\cat{D}(\mrm{H}^0(A_{\bar{\q}}))$.
By Proposition \ref{prop:reduction}(2),
this implies that
\[
\mrm{R}\opn{Hom}_{B_{\bar{\p}}}(A_{\bar{\q}},B_{\bar{\p}}) \cong A_{\bar{\q}}[-n]
\]
in $\cat{D}(A_{\bar{\q}})$, which proves that (\ref{eqn:goalTilting}) holds,
and hence shows that $M$ is invertible.
Next, assume that $\opn{Spec}(\mrm{H}^0(A))$ has connected spectrum.
This implies that the integer $n$ from above is constant, 
and independent of the chosen $\bar{\q}$.
Hence, the complex of $\mrm{H}^0(A)$-modules $M\otimes^{\mrm{L}}_A \mrm{H}^0(A)$ has amplitude $0$,
so it is a shift of an invertible $\mrm{H}^0(A)$-module.
This implies by \cite[Corollary 7.2.2.19]{Lu} that $M$ is a shift of an object in the class $\mathcal{P}$,
in the sense of \cite[Section 2.2]{Mi}.
Using \cite[Theorem 14.1.22]{YeBook}, 
we have that
\[
\mrm{H}^i \left(\mrm{R}\opn{Hom}_B(A,X)\right) \cong
\mrm{H}^i \left( N\otimes^{\mrm{L}}_B X \right) \cong
\mrm{H}^i \left( M\otimes^{\mrm{L}}_A (A\otimes^{\mrm{L}}_B X) \right). 
\]
Let $n = \sup(M)$, and let $\bar{N} = \mrm{H}^n(M)$.
Notice that $\bar{N} \cong \mrm{H}^n(M\otimes^{\mrm{L}}_A \mrm{H}^0(A))$,
so that $\bar{N}$ is an invertible $\mrm{H}^0(A)$-module.
Note further that $M[n]$ belongs to the class $\mathcal{P}$,
so it also belongs to the class $\mathcal{F}$.
Hence, by \cite[Lemma 4.6]{Mi},
we have that
\[
\mrm{H}^i \left( M\otimes^{\mrm{L}}_A (A\otimes^{\mrm{L}}_B X) \right) \cong
\mrm{H}^{i-n} \left( M[n]\otimes^{\mrm{L}}_A (A\otimes^{\mrm{L}}_B X) \right) 
\cong \bar{N}\otimes_{\mrm{H}^0(A)} \mrm{H}^{i-n} \left( A\otimes^{\mrm{L}}_B X \right).
\]
\end{proof}

\section{Homologically smooth maps of flat dimension $0$}\label{sec:smooth}

We now apply the results of the previous section in the following situation.
Suppose that $\varphi:A \to B$ is a map of commutative non-positive DG-rings which is of flat dimension $0$.
As noted above, it give rise to a retraction of DG-rings
\begin{equation}\label{eqn:retractDiagonal}
B \to B\otimes^{\mrm{L}}_A B \to B
\end{equation}
Moreover, note that the assumption that $\varphi$ is of flat dimension $0$ implies by Proposition \ref{prop:flat-base-change} that the map $B \to B\otimes^{\mrm{L}}_A B$ is of flat dimension $0$.

The next definition is due to Kontsevich:
\begin{dfn}
A map $A \to B$ of commutative DG-rings is called homologically smooth 
if it makes $B$ into a perfect object of $\cat{D}(B\otimes^{\mrm{L}}_A B)$.
\end{dfn}

\begin{cor}\label{cor:homSmooth}
Suppose that $\varphi:A \to B$ is a map of commutative non-positive DG-rings which is of flat dimension $0$.
Then $\varphi$ is homologically smooth if and only if $\mrm{H}^0(\varphi):\mrm{H}^0(A) \to \mrm{H}^0(B)$ is homologically smooth.
\end{cor}
\begin{proof}
This follows from applying Theorem \ref{thm:smoothRed} to (\ref{eqn:retractDiagonal}).
\end{proof}

Recall that a if $\k$ is a field, and $A$ is a commutative noetherian $\k$-algebra,
then $A$ is called geometrically regular if $L \otimes_{\k} A$ is a regular ring for any finite field extension $\k \subseteq L$.
A map $A \to B$ between commutative noetherian rings is called regular if it is flat and all of its fibers are geometrically regular.

\begin{cor}\label{cor:homRegular}
Suppose that $\varphi:A \to B$ is a map of commutative non-positive DG-rings which is of flat dimension $0$.
Assume further that $\mrm{H}^0(A)$, $\mrm{H}^0(B)$ and $\mrm{H}^0(B)\otimes_{\mrm{H}^0(A)} H^0(B)$ are noetherian rings.
Then $\varphi$ is homologically smooth if and only if the map $\mrm{H}^0(\varphi):\mrm{H}^0(A) \to \mrm{H}^0(B)$ is regular.
\end{cor}
\begin{proof}
This follows from Corollary \ref{cor:homSmooth} and the fact that a flat map between commutative noetherian rings is homologically smooth if and only if it is regular, a result proved in \cite[Theorem 1]{Ro}.
\end{proof}

A map $A \to B$ between commutative noetherian rings is called essentially of finite type if it is a localization of a finite type map.

\begin{cor}\label{cor:smoothtohzsmooth}
Suppose that $\varphi:A \to B$ is a map of commutative non-positive DG-rings which is of flat dimension $0$.
Assume further that $\mrm{H}^0(A)$ and $\mrm{H}^0(B)$ are noetherian rings,
and that $\mrm{H}^0(\varphi):\mrm{H}^0(A) \to \mrm{H}^0(B)$ is essentially of finite type.
Then $\varphi$ is homologically smooth if and only if the map $\mrm{H}^0(\varphi):\mrm{H}^0(A) \to \mrm{H}^0(B)$ is smooth.
\end{cor}
\begin{proof}
This follows from Corollary \ref{cor:homRegular} and the fact that for essentially finite type maps, being regular is the same as being smooth. See \cite[Theorem 1.1]{AI}.
\end{proof}

\begin{rem}
A map $\varphi:A \to B$ which is of flat dimension $0$,
and such that the induced map $\mrm{H}^0(\varphi):\mrm{H}^0(A) \to \mrm{H}^0(B)$ is smooth,
is called a smooth map in \cite[Section 2.2.2]{TV} and \cite[Section 2.1.4]{GR}.
We thus see that the maps that are called smooth in derived algebraic geometry are exactly the homologically smooth maps of flat dimension $0$.
\end{rem}

\begin{cor}\label{cor:lci}
Suppose that $\varphi:A \to B$ is a map of commutative non-positive DG-rings which is of flat dimension $0$.
Assume further that $\mrm{H}^0(A)$ and $\mrm{H}^0(B)$ are noetherian rings,
and that $\mrm{H}^0(\varphi):\mrm{H}^0(A) \to \mrm{H}^0(B)$ is essentially of finite type.
Then the following are equivalent:
\begin{enumerate}
\item The map $\varphi$ is smooth.
\item The induced map $B\otimes^{\mrm{L}}_A B \to B$ is a local complete intersection:
for any $\bar{\q} \in \opn{Spec}(\mrm{H}^0(B))$,
letting $\bar{\p} \in \opn{Spec}(\mrm{H}^0(B\otimes^{\mrm{L}}_A B))$ be its preimage along the diagonal map,
there exist $\bar{a}_1,\dots,\bar{a}_n \in \mrm{H}^0((B\otimes^{\mrm{L}}_A B)_{\bar{\p}})$, 
such that the composition
\[
B_{\bar{\q}} \xrightarrow{\varphi_{\bar{\p}}} (B\otimes^{\mrm{L}}_A B)_{\bar{\p}} \xrightarrow{\kappa} K((B\otimes^{\mrm{L}}_A B)_{\bar{\p}};\bar{a}_1,\dots,\bar{a}_n)
\]
is a quasi-isomorphism.
\end{enumerate}
If moreover the DG-rings $A,B$ are noetherian and have bounded cohomology,
then the sequence $\bar{a}_1,\dots,\bar{a}_n \in \mrm{H}^0((B\otimes^{\mrm{L}}_A B)_{\bar{\p}})$ is a $(B\otimes^{\mrm{L}}_A B)_{\bar{\p}}$-regular sequence.
\end{cor}
\begin{proof}
Let us denote the map $B\otimes^{\mrm{L}}_A B \to B$ by $\Delta$.
By Corollary \ref{cor:smoothtohzsmooth} the map $\varphi$ is smooth if and only if the map $\mrm{H}^0(\varphi)$ is smooth,
and by \cite[Theorem 1.1]{AI}, this is equivalent to the map 
\[
\mrm{H}^0(\Delta):\mrm{H}^0(B)\otimes_{\mrm{H}^0(A)} \mrm{H}^0(B) \to \mrm{H}^0(B)
\]
being a local complete intersection.
By Theorem \ref{thm:lciRed},
this is equivalent to (2) above.
Under the additional assumption that $A,B$ are noetherian and with bounded cohomology,
our assumptions imply that $B\otimes^{\mrm{L}}_A B$ is also noetherian and has bounded cohomology,
so $(B\otimes^{\mrm{L}}_A B)_{\bar{\p}}$-regularity of $\bar{a}_1,\dots,\bar{a}_n$ follows from Proposition \ref{prop:isRegular}.
\end{proof}

Next we have the following version of Van den Bergh duality (\cite[Theorem 1]{VdB}):
\begin{cor}\label{cor:VDB}
Suppose that $\varphi:A \to B$ is a map of commutative non-positive DG-rings which is of flat dimension $0$.
Assume further that $\mrm{H}^0(A)$ and $\mrm{H}^0(B)$ are noetherian rings,
and that $\mrm{H}^0(\varphi):\mrm{H}^0(A) \to \mrm{H}^0(B)$ is essentially of finite type.
If $\varphi$ is smooth then the following holds:
\begin{enumerate}
\item Letting $M:= \mrm{R}\opn{Hom}_{B\otimes^{\mrm{L}}_A B}(B,B\otimes^{\mrm{L}}_A B) \in \cat{D}(B)$,
we have that $M$ is an invertible DG-module,
and for any $X \in \cat{D}(B\otimes^{\mrm{L}}_A B)$,
there are natural isomorphisms
\[
\mrm{R}\opn{Hom}_{B\otimes^{\mrm{L}}_A B}(B,X) \cong M\otimes^{\mrm{L}}_B (B\otimes^{\mrm{L}}_{B\otimes^{\mrm{L}}_A B} X)
\]
and
\[
M^{-1}\otimes^{\mrm{L}}_B \mrm{R}\opn{Hom}_{B\otimes^{\mrm{L}}_A B}(B,X) \cong  B\otimes^{\mrm{L}}_{B\otimes^{\mrm{L}}_A B} X
\]
\item If moreover the noetherian ring $\mrm{H}^0(B)$ has connected spectrum,
then there exist an invertible $\mrm{H}^0(B)$-module $\bar{M}$ and an integer $n$,
such that for all $i \in \mathbb{Z}$ and all $X \in \cat{D}(B\otimes^{\mrm{L}}_A B)$ there is a natural isomorphism
\[
\opn{Ext}^i_{B\otimes^{\mrm{L}}_A B}(B,X) \cong \bar{M}\otimes_{\mrm{H}^0(B)} \opn{Tor}_{n-i}^{B\otimes^{\mrm{L}}_A B}(B,X).
\]
\end{enumerate}
\end{cor}
\begin{proof}
This follows from Theorem \ref{thm:VDB}.
\end{proof}

\begin{cor}\label{cor:rigid}
Suppose that $\varphi:A \to B$ is a map of commutative non-positive DG-rings which is of flat dimension $0$.
Assume further that $\mrm{H}^0(A)$ and $\mrm{H}^0(B)$ are noetherian rings,
and that $\mrm{H}^0(\varphi):\mrm{H}^0(A) \to \mrm{H}^0(B)$ is essentially of finite type.
If $\varphi$ is smooth then there exist an invertible DG-module $N \in \cat{D}(B)$ such that
\[
\mrm{R}\opn{Hom}_{B\otimes^{\mrm{L}}_A B}(B,N\otimes^{\mrm{L}}_A N) \cong N.
\]
\end{cor}
\begin{proof}
Let $M:= \mrm{R}\opn{Hom}_{B\otimes^{\mrm{L}}_A B}(B,B\otimes^{\mrm{L}}_A B)$.
By Corollary \ref{cor:VDB}, we know that $M$ is invertible over $B$.
Let $N = M^{-1}$, so $N$ is also invertible over $B$.
Then by Corollary \ref{cor:VDB} we have that:
\[
\mrm{R}\opn{Hom}_{B\otimes^{\mrm{L}}_A B}(B,N\otimes^{\mrm{L}}_A N) \cong M\otimes^{\mrm{L}}_B \left(B\otimes^{\mrm{L}}_{B\otimes^{\mrm{L}}_A B} (N\otimes^{\mrm{L}}_A N)\right).
\]
The reduction to the diagonal isomorphism (see for instance \cite[Equation (3.11.2)]{AILN}) says that
\[
B\otimes^{\mrm{L}}_{B\otimes^{\mrm{L}}_A B} (N\otimes^{\mrm{L}}_A N) \cong N\otimes^{\mrm{L}}_B N.
\]
This implies that
\[
\mrm{R}\opn{Hom}_{B\otimes^{\mrm{L}}_A B}(B,N\otimes^{\mrm{L}}_A N) \cong M\otimes^{\mrm{L}}_B N\otimes^{\mrm{L}}_B N \cong N,
\]
proving the claim.
\end{proof}

\begin{rem}
Following \cite[Definition 9.2]{Ye1},
given a map of commutative DG-rings $A\to B$,
a DG-module $N \in \cat{D}(B)$ together with an isomorphism
\[
\rho:N \to \mrm{R}\opn{Hom}_{B\otimes^{\mrm{L}}_A B}(B,N\otimes^{\mrm{L}}_A N)
\]
is called a rigid DG-module over $B$ relative to $A$.
Thus, the above corollary states that in the above smooth situation, 
there exist an invertible DG-module over $B$ which is rigid over $B$ relative to $A$.
In this smooth context, one may view this result as a generalization of similar results from \cite{AIL,Sh,YeSquare,YZ}.
\end{rem}

\section{Homological smoothness and flatness over a base ring}\label{sec:smooth-hz-flat}

In this final section we study homologically smooth commutative DG-rings $A$ over a base ring $\k$.
Here, since the base $\k$ is a ring, it makes no sense to impose a flat dimension $0$ assumption,
as this would imply that $\amp(A) = 0$.
Instead, we will assume that the induced map $\k \to \mrm{H}^0(A)$ is flat.
This is the case, for instance, if $\k$ is a field.
Our main result below shows that under a noetherian assumption,
commutative DG-rings which are not equivalent to a ring are never homologically smooth over a field.

\begin{thm}\label{thm:flathz}
Let $\k$ be a commutative noetherian ring,
and let $A$ be a non-positive commutative noetherian DG-ring over $\k$.
Assume that:
\begin{enumerate}
\item The map $\k \to A$ is homologically smooth.
\item The DG-ring $A$ is noetherian, it has bounded cohomology, and the ring $A^0$ is noetherian.
\item The map $\k \to \mrm{H}^0(A)$ obtained from the composition $\k \to A \to \mrm{H}^0(A)$ is flat. 
\end{enumerate}
Then the natural map $A \to \mrm{H}^0(A)$ is a quasi-isomorphism of DG-rings.
\end{thm}
\begin{proof}
The assumption that $A$ is homologically smooth over $\k$ means that
$A$ is a perfect DG-module over $A\otimes^{\mrm{L}}_{\k} A$.
By Proposition \ref{prop:compact},
this is equivalent to the fact that
\[
A \otimes^{\mrm{L}}_{A\otimes^{\mrm{L}}_{\k} A} \mrm{H}^0(A\otimes^{\mrm{L}}_{\k} A)
\]
is a perfect complex over the ring $\mrm{H}^0(A\otimes^{\mrm{L}}_{\k} A)$.
Since $A$ is non-positive, by the K\"unneth Trick (\cite[Lemma 13.1.36]{YeBook}),
we have that 
\[
\mrm{H}^0(A\otimes^{\mrm{L}}_{\k} A) \cong \mrm{H}^0(A)\otimes_{\k} \mrm{H}^0(A),
\]
and since we assumed that $\mrm{H}^0(A)$ is flat over $\k$,
this is isomorphic to $\mrm{H}^0(A)\otimes^{\mrm{L}}_{\k} \mrm{H}^0(A)$.
Thus, we deduce that the complex
\[
A \otimes^{\mrm{L}}_{A\otimes^{\mrm{L}}_{\k} A} \left(\mrm{H}^0(A)\otimes^{\mrm{L}}_{\k} \mrm{H}^0(A)\right)
\]
is a perfect complex over the ring $\mrm{H}^0(A)\otimes_{\k} \mrm{H}^0(A)$.
By the reduction to the diagonal isomorphism (\cite[Equation (3.11.2)]{AILN} or \cite[Proposition 5.3]{Sh}),
we may compute the above derived Hochschild homology,
and obtain a natural isomorphism
\[
A \otimes^{\mrm{L}}_{A\otimes^{\mrm{L}}_{\k} A} \left(\mrm{H}^0(A)\otimes^{\mrm{L}}_{\k} \mrm{H}^0(A)\right)
\cong \mrm{H}^0(A)\otimes^{\mrm{L}}_A \mrm{H}^0(A)
\]
in $\cat{D}(\mrm{H}^0(A)\otimes_{\k} \mrm{H}^0(A))$.
In particular,
the complex $\mrm{H}^0(A)\otimes^{\mrm{L}}_A \mrm{H}^0(A)$ has finite flat dimension over $\mrm{H}^0(A)\otimes_{\k} \mrm{H}^0(A)$.
Since $\mrm{H}^0(A)$ is flat over $\k$,
we know that $\mrm{H}^0(A)\otimes_{\k} \mrm{H}^0(A)$ is flat over $\mrm{H}^0(A)$,
so we may deduce from the above that the complex $\mrm{H}^0(A)\otimes^{\mrm{L}}_A \mrm{H}^0(A)$ has finite flat dimension over $\mrm{H}^0(A)$.
But it also has finitely generated cohomology over $\mrm{H}^0(A)$,
and $\mrm{H}^0(A)$ is a noetherian ring,
so we may deduce that $\mrm{H}^0(A)\otimes^{\mrm{L}}_A \mrm{H}^0(A)$ is a perfect complex over $\mrm{H}^0(A)$.
Applying Proposition \ref{prop:compact} again,
we deduce that $\mrm{H}^0(A)$ is a perfect DG-module over $A$.
By \cite[Theorem 7.21]{Ye1} (or \cite[Theorem 0.2]{Jo} in case $\mrm{H}^0(A)$ is local),
this implies that the map $A \to \mrm{H}^0(A)$ is a quasi-isomorphism.
\end{proof}

\begin{cor}
Assume the conditions of the above theorem,
and assume further that the ring $\k$ is a regular ring.
Then $A \cong \mrm{H}^0(A)$ is also a regular ring.
\end{cor}
\begin{proof}
We may replace $A$ by $\mrm{H}^0(A)$ and assume that $A$ is a ring.
We know that $A$ is flat and homologically smooth over $\k$,
which implies by \cite[Theorem 1]{Ro} that the map $\k \to A$ is a regular morphism.
Since $\k$ is a regular ring, 
this implies that $A$ is a regular ring.
\end{proof}

\begin{cor}
Let $\k$ be a field
and let $A$ be a non-positive homologically smooth commutative noetherian DG-ring over $\k$,
such that $A$ has bounded cohomology and $A^0$ is a noetherian ring.
Then $A$ is quasi-isomorphic to a regular ring.
\end{cor}

Using the generic flatness theorem, we may deduce from this that noetherian DG-rings that are homologically smooth over a integral domain are quasi-isomorphic to a ring on a non-empty open set.

\begin{cor}
Let $\k$ be a noetherian integral domain,
and let $A$ be a noetherian DG-ring over $\k$ with bounded cohomology such that $A^0$ is a noetherian ring.
Assume that $A$ is homologically smooth over $\k$,
and that the induced map $\k \to \mrm{H}^0(A)$ is of finite type.
Then the set
\[
W = \{\bar{\p} \in \opn{Spec}(\mrm{H}^0(A)) \mid A_{\bar{\p}} \cong \mrm{H}^0(A_{\bar{\p}})\}
\]
is a non-empty open set.
\end{cor}
\begin{proof}
Let $n = -\inf(A) = \amp(A)$.
For each $k \in \mathbb{N}$,
consider the set
\[
V_k = \{\bar{\p} \in \opn{Spec}(\mrm{H}^0(A)) \mid \mrm{H}^k(A_{\bar{\p}}) \ne 0\}.
\]
The noetherian assumption on $A$ implies that $\mrm{H}^k(A)$ is a finitely generated $\mrm{H}^0(A)$-module,
so the set $V_k$, being the support of a finitely generated module, is closed.
Since 
\[
W = \opn{Spec}(\mrm{H}^0(A)) \setminus \left(V_{-1} \cup V_{-2} \cup \dots \cup V_{-n}\right),
\]
we see that $W$ is open.
On the other hand, 
denoting the map $\k \to \mrm{H}^0(A)$ by $\varphi$,
we know from \cite[tag 051R]{SP} that the set 
\[
\{\bar{\p} \in \opn{Spec}(\mrm{H}^0(A)) \mid \k_{\varphi^{-1}(\bar{\p})} \to \mrm{H}^0(A)_{\bar{\p}} \mbox{ is flat}\}
\]
is non-empty.
Hence, by Theorem \ref{thm:flathz}, the set $W$ is also non-empty.
\end{proof}

\textbf{Acknowledgments.}

The author thanks Amnon Yekutieli for helpful discussions.
The author is thankful to an anonymous referee for several corrections that helped significantly improving this manuscript.
This work has been supported by the Charles University Research Centre program No.UNCE/SCI/022,
and by the grant GA~\v{C}R 20-02760Y from the Czech Science Foundation.

\end{document}